\documentclass{article}[11]
\usepackage{amsmath, amsthm, amssymb, color, amsrefs, authblk}

\newtheorem{thm}{Theorem}[section]
\newtheorem*{thm*}{Theorem}
\newtheorem{prop}[thm]{Proposition}
\newtheorem{lem}[thm]{Lemma}
\newtheorem{cor}[thm]{Corollary}
\newtheorem{definition}[thm]{Definition}

\DeclareMathOperator\vol{vol}

\DeclareMathOperator\GL{GL}

\DeclareMathOperator\Hess{Hess}
\DeclareMathOperator\ver{ver}
\DeclareMathOperator\hor{hor}
\DeclareMathOperator\Sym{Sym}

\DeclareMathOperator\Ric{Ric}

\DeclareMathOperator\Imm{Im}
\DeclareMathOperator\Vol{Vol}

\def\div{\mbox{div}}

\def\i{\mbox{i}}

\DeclareMathOperator\tr{tr}
\def\C{\mathbb{C}}

\def\CP{\mathbb{CP}}
\def\Sph{\mathbb{S}}

\def\R{\mathbb{R}}

\def\Z{\mathbb{Z}}
\def\T{\mathbb{T}}

\begin{document}

\title{Pluri-potential theory, submersions\\and calibrations}

\author[*]{Tommaso Pacini}
\affil[*]{Department of Mathematics, University of Torino \newline via Carlo Alberto 10, 10123 Torino, Italy \newline tommaso.pacini@unito.it}


\maketitle

\begin{abstract}
We present a systematic collection of results concerning interactions between convex, subharmonic and pluri-subharmonic functions on pairs of manifolds related by a Riemannian submersion. Our results are modelled on those known in the classical complex-analytic context and represent another step in the recent Harvey-Lawson pluri-potential theory for calibrated manifolds. In particular we study the case of K\"ahler and $G_2$ manifolds, emphasizing both parallels and differences. We show that previous results concerning Lagrangian fibrations can be viewed as an application of this framework.
\end{abstract}

\section{Introduction}\label{s:intro}
The official goal of this paper is to present a systematic collection of results concerning relationships between different potential-theoretic function theories on Riemannian manifolds $M$, $B$ related by a Riemannian submersion $\pi:M\rightarrow B$. More specifically, we will be concerned with various classes of convex, subharmonic and pluri-subharmonic functions on $M$ and with the construction of corresponding convex or subharmonic functions on $B$.

The most classical example of this type of question arises in the case 
$$\pi:\R^2\setminus\{0\}\rightarrow \R^+, \ \ (r,\theta)\mapsto r,$$
viewed as a Riemannian submersion with $\Sph^1$-fibres.
Subharmonic $\Sph^1$-invariant functions above then correspond to convex functions below (wrt the variable $\log r$). More generally, the 3-circle theorem of Hadamard (respectively, a theorem of Hardy) states that the fibre-wise supremum (respectively, the fibre-wise integral) of any subharmonic function above defines a convex function below (wrt the variable $\log r$). These results not only clarify the content of the subharmonic condition, but are also very useful in applications.

Our results in sections \ref{s:invariant}, \ref{s:integralsup} can be seen as a broad generalization of these classical theorems.

\ 

More generally, our interest is in the following points, whose unifying thread is to explore the geometric framework underlying the above function-theoretic question.

\paragraph{Immersions vs. submersions.}Potential theory provides a means, within geometry, for controlling minimal submanifolds. The prime example of this is the standard proof that $\R^n$ contains no compact minimal submanifolds. This result concerns potential theory and immersions. Dualizing, we are interested in (pluri-)potential theory and submersions, and in possible geometric applications of these results. 

Our current main application is Theorem \ref{thm:kahlersubs}, concerning certain Lagrangian fibrations and the volume of the fibres. This result should be viewed as yet another manifestation of the relationship between volume and Ricci curvature, valid in very general Riemannian contexts. The Bishop-Gromov theorem is the most famous example of this relationship.

\paragraph{Pluri-potential theories.}Classical, ie complex-analytic, pluri-potential theory was developed as a tool for studying holomorphic functions. Recent work by Harvey-Lawson, eg \cite{HL3}, has highlighted the existence of analogous pluri-potential theories in many new geometric contexts. In particular, given any calibrated manifold, there exists a notion of pluri-subharmonic (PSH) functions specific to the geometry of that manifold \cite{HL2}. A surprising feature of these theories is the extent of their parallelisms with the classical theory. Our results push this study a step forward, emphasizing that such parallelisms continue to hold in the context of Riemannian submersions. 

The Harvey-Lawson theory is also surprising in the breadth of its applicability. For the sake of definiteness we instead choose to concentrate on few specific examples, chosen from opposite sides of the spectrum: the Riemannian and K\"ahler theories on the more classical side, $G_2$ theory on the exotic side. 

\paragraph{Submanifold geometry.}Our presentation emphasizes, in particular, the interplay between ``opposite" classes of submanifolds: in the language of Harvey-Lawson, these are the calibrated and the ``free" submanifolds. We are thus interested in complex/Lagrangian submanifolds in the K\"ahler case, associative/coassociative submanifolds in the $G_2$ case. 

Our results also rely on second variation formulae specifically tailored to the submanifolds in question. These ingredients would play a similar role in any other calibrated pluri-potential theory.

\paragraph{Geometric constructions of PSH functions.}Although convex, subharmonic and PSH functions exist (locally) in abundance, it is an interesting question to find geometrically meaningful examples. For example, given a point $p$ in a Riemannian manifold $M$, let $r^2$ denote the square-distance function from $p$. It is a standard fact that this function is convex in a neighbourhood of $p$. In the context of Riemannian submersions, the volume functional on the fibres can be used to construct functions which are convex or PSH at certain points: see Corollary \ref{cor:volume}. Yet another such construction, in K\"ahler geometry, is at the heart of Theorem \ref{thm:kahlersubs}.

\paragraph{K\"ahler vs. $G_2$ theory.}Superficially, there are intriguing parallels between K\"ahler and $G_2$ geometry. It is an interesting, actively pursued, question to what extent these parallels hold at a deeper level. 

K\"ahler, more generally complex, geometry has been investigated for centuries and classical pluri-potential theory has found many geometric applications; our Theorem \ref{thm:kahlersubs} relies on the well-developed relationship between curvature and holomorphic line bundles. 

By contrast, $G_2$ geometry is still in its infancy. Its pluri-potential theory still has basically no geometric applications, and there is currently no known analogue of the curvature-vector bundle relationship. Theorem \ref{thm:kahlersubs} is thus interesting also because it marks the point, in this paper, in which the K\"ahler-$G_2$ parallism breaks down. It would be very interesting to find analogues of Theorem \ref{thm:kahlersubs} in the $G_2$ context.

\paragraph{Final comments.}The existence of deep relationships between classical \mbox{(pluri-)} potential theory and Differential Geometry is nowadays clear. Convex functions are related to the properties of geodesics, to curvature, to convex sets and to the geometric structure of Riemannian manifolds, see eg \cite{GreeneShiohama},\cite{Eschenburg},\cite{Sakai}. PSH functions are related to Stein manifolds and to the search for special K\"ahler metrics. In this paper we focus on other aspects, outlined above, of these relationships. Hopefully, our results will stimulate research into relationships between the new Harvey-Lawson pluri-potential theories and the geometry of the corresponding ambient spaces.
 
\ 

The paper is structured as follows. The main results appear in sections \ref{s:invariant} and \ref{s:integralsup}. The main applications of these results appear as corollaries there, and in section \ref{s:application}. Sections \ref{s:hessian}-\ref{s:HL} and section \ref{s:secondvar} provide the necessary background. Section \ref{s:fibrations} is a digression on Lagrangian fibrations in Symplectic geometry and on coassociative fibrations in $G_2$ geometry.

\ 

\textit{Remark. }All our submanifolds will be without boundary. Our geometric applications allow us to mostly restrict to smooth PSH functions. This choice also helps us to stay focused on the main goals of the paper. Standard techniques of pluri-potential theory, as extended by Harvey-Lawson \cite{HL3}, would allow weaker notions of PSH functions, based on upper semi-continuous functions.

\ 

\textit{Acknowledgements. }Potential and pluri-potential theory, and their geometric applications, are a vast subject. I have provided references to every relevant result I know of. I am not aware of any serious overlap between the results of sections \ref{s:invariant}, \ref{s:integralsup} and the literature. Section \ref{s:application} provides a detailed comparison between Theorem \ref{thm:kahlersubs} and previously-known results. I am happy to thank Anna Fino for interesting conversations and the anonymous referee for useful comments.

\section{The Hessian operator}\label{s:hessian}

Given a Riemannian manifold $M$, let $\nabla$ denote its Levi-Civita connection. Recall the following definitions.
\begin{definition}
The Hessian is the operator 
$$\Hess(f)(X,Y):=X(Yf)-\nabla_XY(f).$$ 
The Laplacian is the operator $\Delta f:=\tr\Hess(f).$

A function $f$ is convex, respectively subharmonic, at a point $p\in M$ if $\Hess(f)\geq 0$, respectively $\Delta f\geq 0$, at that point.
\end{definition}
Notice that the correction term $\nabla_XY(f)$ vanishes at critical points, where the Hessian is thus independent of the metric.

We are interested in how the Hessian interacts with immersions and submersions. Let us view these in turn.

\paragraph{Immersions.}Let $(M,g)$ be a Riemannian manifold and $\Sigma\hookrightarrow M$ be a submanifold endowed with the restricted metric. Each has its own Hessian. In order to compare them, choose $X,Y\in T_p\Sigma$ (locally extended). Then, using standard notation,
\begin{align*}
\Hess_M(f)(X,Y)&=X(Y(f))-(\nabla_XY)^T(f)-(\nabla_XY)^\perp(f)\\
&=\Hess_{\Sigma}(f)(X,Y)-(\nabla_XY)^\perp(f).
\end{align*}
The two Hessians are thus related by the second fundamental form of $\Sigma$. This leads to the following conclusion.
\begin{prop}\label{prop:sigma_restrict}
Let $\Sigma\hookrightarrow (M,g)$ be a submanifold endowed with the restricted metric.
\begin{enumerate}
\item Assume that either $df_{|T\Sigma^\perp}=0$ or $\Sigma$ is totally geodesic. Then $\Hess_{\Sigma}(f)=\Hess_M(f)_{|T\Sigma}$.
\item Assume $\Sigma$ is minimal. Then $\Delta_\Sigma f=\tr_{|T\Sigma} \Hess_M(f)$. In particular, if $\Hess_M(f)\geq 0$ then $f_{|\Sigma}$ is subharmonic.
\end{enumerate}
\end{prop}

\textit{Remark. }Simple examples, such as $f(x,y,z):=x^2+y^2-z^2$ and $\Sigma:=z-\mbox{axis}\subset\R^3$, show that in general one cannot hope that subharmonic functions restrict to subharmonic.

\ 

This relationship implies that potential theory provides strong control over minimal submanifolds. For example:

1. $\R^n$ does not contain compact minimal submanifolds. Indeed, it suffices to apply the proposition to each coordinate function $f:=x_i$: the ambient Hessian vanishes so each restricted function would be harmonic on the submanifold, thus constant.

2. Assume $\Sigma\subseteq M$ is compact and minimal, and that $\Hess_M(f)$ is non-negative. Then $f_{|\Sigma}$ is constant, ie $\Sigma$ is contained in a level set of $f$. More-over, $\Sigma$ cannot intersect any region where $\Hess_M(f)$ is positive definite: otherwise $\Delta_\Sigma f>0$, contradicting the fact that $f_{|\Sigma}$ is constant \cite{TsaiWang}.

The existence of an appropriate ``background $f$" thus puts strong restrictions on the properties and location of minimal submanifolds.

\ 

\textit{Remark. }Part 1 above depends of course strongly on the chosen ambient metric: stereographic projection from $\Sph^2$ produces a metric on $\R^2$ which does admit a minimal $\Sph^1$.

\paragraph{Submersions: generalities.}In Differential Topology the natural dual of the notion of immersion is that of submersion. In our context we shall be interested in the corresponding Riemannian notion. It may be useful to review some basic facts, as follows.

Let $M$, $B$ be Riemannian manifolds and $\pi:M\rightarrow B$. This data is called a Riemannian submersion if $\pi$ is a surjective map and each restriction $d\pi:\ker(d\pi)^\perp\rightarrow TB$ is an isometry. The tangent bundle $TM$ splits into the direct sum of two natural distributions: $\ker(d\pi)$, known as the vertical distribution, and $\ker(d\pi)^\perp$, known as the horizontal distribution. The vertical distribution is automatically integrable. We will let $X=X^{\ver}+X^{\hor}$ denote the corresponding decomposition of a vector. 

A horizontal vector field $X$ on $M$ is called basic if its projection onto $B$ is a well-defined vector field on $B$, ie if, along each fibre $\pi^{-1}(b)$, all vectors $d\pi(X)$ coincide. We will often rely on the corresponding 1:1 identification
$$\mbox{vector fields on $B$}\leftrightarrow\mbox{horizontal basic vector fields on $M$},$$
and more generally on all identifications between horizontal/$\pi$-invariant objects above and the corresponding objects below.

It is known \cite{Hermann}, \cite{Nagano} that $M$ complete implies $B$ complete, and that in this case all fibres are diffeomorphic and $M$ is a locally trivial fibre bundle. Furthermore, if the fibres are totally geodesic then all fibres are isometric and the structure group of the fibre bundle is given by isometries. This does not yet imply that $M$ is locally trivial from the Riemannian viewpoint, ie that $M$ is locally isometric to the Riemannian product $B\times F$, where $F$ is the generic fibre: in this case the horizontal distribution can be identified with $TB$, so it must be integrable. 

These conditions are encoded by two tensors $T$, $A$ on $M$ introduced in \cite{Oneill}. The tensor $T$ is defined as an extension of the second fundamental form 
\begin{equation*}
\ker(d\pi)\times \ker(d\pi)\rightarrow \ker(d\pi)^\perp, \ \ (X,Y):=(\nabla_XY)^\perp,
\end{equation*}
of the fibres; it vanishes iff the fibres are totally geodesic. The tensor $A$ can also be defined using $\nabla$; equivalently, it is an extension of the tensor 
\begin{equation*}
\ker(d\pi)^\perp\times \ker(d\pi)^\perp\rightarrow \ker(d\pi), \ \ (X,Y):=[X,Y]^{\ver}.
\end{equation*}
It vanishes iff the horizontal distribution is integrable.

It follows from \cite{Oneill} that both tensors $T$, $A$ vanish iff $M$ is a locally trivial fibre bundle from the Riemannian viewpoint, as defined above.

\ 

We will use the following general fact. 

\begin{lem}
Let $\pi:M\rightarrow B$ be a Riemannian submersion. 
\begin{enumerate}
\item Let $\alpha$ be a curve in $B$. Then a horizontal lift of $\alpha$ is a geodesic in $M$ iff $\alpha$ is a geodesic in $B$.
\item Any geodesic in $M$ which is initially horizontal is always horizontal, and is thus the lift of a geodesic in $B$.
\end{enumerate}
\end{lem}
\begin{proof}
Part 1: Recall, cf eg \cite{Petersen}, that if $X$, $Y$ are vector fields on $B$ then, using the same notation for their basic horizontal lifts, the covariant derivatives are related by the general formula
$$\nabla^M_XY=\nabla^B_XY+(1/2)[X,Y]^{\ver}.$$
Thus $\nabla^M_{\dot\alpha}\dot\alpha=\nabla^B_{\dot\alpha}\dot\alpha+(1/2)[\dot\alpha,\dot\alpha]^{\ver}=\nabla^B_{\dot\alpha}\dot\alpha$.

Part 2: Let us project the initial point and direction of the geodesic down to $B$. We then generate a geodesic in $B$, which by Part 1 lifts to a horizontal geodesic in $M$. It has the same initial data as the original geodesic, so the two coincide.
\end{proof}

\ 

\textit{Examples. }Examples of Riemannian submersions include products, for which both tensors $T$, $A$ vanish, and warped products, in which all fibres are diffeomorphic but not necessarily isometric (so $T\neq 0$).

Below we will be interested in the class of examples provided by manifolds with an isometric group action. A special case is that of principal fibre bundles over a Riemannian manifold: a connection on the bundle, together with an invariant metric on the Lie algebra, then generates a Riemannian submersion. In this case the vanishing of $A$, ie the integrability of the horizontal distribution defined by the connection, is equivalent to the vanishing of the connection's curvature.

\ 

\textit{Example. }The following construction provides an example of a Riemannian submersion with a totally geodesic fibre such that $A=0$ only along the fibre. This is relevant to Proposition \ref{p:potential}, below.

Consider the manifold $\Sph^1\times\R^2$ endowed with the standard product Riemannian structure $g$. Using variables $(\theta,x,y)$ we can identify $g$ with the matrix $Id$. The standard projection onto $\R^2$ is a Riemannian submersion. It has totally geodesic fibres and the horizontal distribution is integrable.

Choose symmetric matrices of the form
$$h(x,y)=\left(\begin{array}{ccc}
0&a&b\\
a&0&0\\
b&0&0
\end{array}\right),$$
where $a=a(x,y)$, $b=b(x,y)$. Fix $\epsilon$ small. Then $g':=Id+\epsilon h$ defines (locally) a new metric on $\Sph^1\times\R^2$. The fact that $a,b$ do not depend on $\theta$ implies that $g'$ is $\Sph^1$-invariant, so the standard projection again defines a Riemannian submersion. 

Consider the map
$$c:\Sph^1\times\R^2\rightarrow\Sph^1\times\R^2,\ \ c(\theta,x,y):=(\theta,-x,-y).$$
Notice that $c^2=Id$. If we assume $a(-x,-y)=-a(x,y)$, $b(-x,-y)=-b(x,y)$ then $g'$ is $c$-invariant, ie $c$ is a $g'$-isometry. One can use this (or a direct calculation) to show that its set of fixed points, ie the fibre $\Sph^1\times\{(0,0)\}$, is totally geodesic. 

Vector fields which are horizontal, ie $g'(v,\partial\theta)=0$, take the form 
$$v(\theta,x,y)=(-\epsilon a v_1-\epsilon b v_2,v_1,v_2).$$
Choosing two such vector fields $v,w$, we can compute that
$$[v,w]^{ver}=\left(v_1(-\epsilon b_x+\epsilon a_y)w_2+v_2(-\epsilon a_y+\epsilon b_x)w_1\right)\partial\theta.$$
Notice that this expression is tensorial in $v,w$, as expected.

The horizontal distribution is integrable at a given point iff $a_y=b_x$ there. If, for example, we choose $a=x^2y$ and $b=-xy^2$ then integrability holds only along the fibre $\Sph^1\times\{(0,0)\}$, as desired.

A final remark: at any point where $(a,b)=(0,0)$, the horizontal spaces defined by $g$, $g'$ coincide. At any other point, the vector $v:=(0,-b,a)$ is horizontal wrt both metrics; this corresponds to the fact that, for dimensional reasons, the two planes must intersect. One finds
$$|v|_g=|v|_{g'}=\sqrt{a^2+b^2}.$$ 
The vector $w:=(-\epsilon a^2-\epsilon b^2, a,b)$ is also $g'$-horizontal. One finds that $g'(v,w)=0$ and that 
$$|w|_{g'}=\sqrt{a^2+b^2}\cdot\sqrt{1-\epsilon^2(a^2+b^2)}.$$
The normalized projections $(|v|_{g'})^{-1}(-b,a)$, $(|w|_{g'})^{-1}(a,b)$ are an orthonormal frame on the base space $\R^2$, wrt the metric induced by $g'$.

\paragraph{Submersions: functions.}Given a Riemannian submersion $(M,B,\pi)$, we shall be interested in three classes of functions:
\begin{itemize}
\item Invariant functions $f:M\rightarrow\R$, ie those of the form $f=\pi^*F$, for some $F:B\rightarrow\R$. 
\item Fibre-wise integral functions $F:B\rightarrow\R$, ie those of the form 
$$F(b)=\int_{\pi^{-1}(b)}f\vol,$$ 
for some $f:M\rightarrow\R$. We will assume the fibres are compact. Standard results ensure the smoothness of $F$ wrt the variable $b$. 
\item Fibre-wise supremum functions $F:B\rightarrow\R$, ie those of the form 
$$F(b)=\sup\{f(p):p\in \pi^{-1}(b)\},$$ 
for some $f:M\rightarrow\R$. We will assume the fibres are compact. This also serves to ensure local (wrt the variable $b$) uniform continuity of $f$, thus equicontinuity of the corresponding family (wrt the parameter defined by points $y$ on the fibre) of functions $b\mapsto f(b,y)$. In turn, this ensures the continuity of $F$. This will suffice for our purposes, though further assumptions would guarantee smoothness. 
\end{itemize}

The following result is the natural dual of Proposition \ref{prop:sigma_restrict}. It presents the most basic relationship between the Hessians of invariant functions. We will later generalize this to other classes of functions.

\begin{prop}\label{prop:hessian}
Let $\pi:M\rightarrow B$ be a Riemannian submersion and $f=\pi^*F$ be an invariant function. Then 
$$\Hess_M(f)(X,X)=\Hess_B(F)(X,X),$$ 
for all horizontal vectors $X$.
\end{prop}
\begin{proof}
One way to prove this is by choosing a geodesic $\alpha$ with initial velocity $X$. Then
$$\Hess_M(f)(X,X)=\frac{d^2}{dt^2}(f\circ\alpha)_{|t=0}=\frac{d^2}{dt^2}(F\circ\alpha)_{|t=0}=\Hess_B(F)(X,X).$$
Alternatively, wrt any two horizontal basic vector fields $X$, $Y$, 
\begin{align*}
\Hess_M(f)(X,Y)&=\nabla_X\nabla_Yf-\nabla_XY(f)\\
&=\nabla_X\nabla_Yf-(\nabla_XY)^{\hor}(f)-(1/2)[X,Y]^{\ver}(f)\\
&=\nabla_X\nabla_YF-\nabla_XY(F)=\Hess_B(F)(X,Y),
\end{align*}
where we use the fact that $f$ is constant on fibres.
\end{proof}

\section{Function theory and submersions: Model cases}\label{s:models}

Given a Riemannian submersion $\pi:M\rightarrow B$, we are interested in the natural next step beyond Proposition \ref{prop:hessian}: relating function theory on $M$ to function theory on $B$. The exact relationship will depend on the function theories involved. Let us review two classical situations, which will serve as models.

\paragraph{Potential theory.}The submersion structure provides a very strong relationship between horizontal directions in $M$ and $B$. In order to get a relationship between function theories, however, it is necessary to gain control over the vertical directions. We can do this via additional assumptions either on the function or on the extrinsic geometry of the fibre, as follows.

\begin{prop}\label{p:potential}
Let $\pi:M\rightarrow B$ be a Riemannian submersion. Assume $f=\pi^*F$. Choose a fibre $\Sigma_b$. Then, at each point of $\Sigma_b$,
\begin{enumerate}
\item If $f$ is convex then $F$ is convex.
\item If either
\begin{itemize}
\item[(i)] $\Sigma_b$ is contained in the critical locus of $f$, or 
\item[(ii)] $T=0$ and $A=0$ at each point of $\Sigma_b$ 
\end{itemize}
then $\Hess_M(f):=\pi^*(\Hess_B(F))$.

In particular, $F$ is convex at $b\in B$ iff $f$ is convex at each point of $\Sigma_b$.
\item If $\Sigma_b$ is minimal then $\Delta_Mf=\pi^*(\Delta_BF)$. 

In particular, $F$ is subharmonic at $b\in B$ iff $f$ is subharmonic at each point of $\Sigma_b$.
\end{enumerate}
\end{prop}
\begin{proof}
Part 1 is a direct consequence of Proposition \ref{prop:hessian}.
 
Part 2: Assume $X$ horizontal, $Y$ vertical. We must examine three cases.

In horizontal directions, by Proposition \ref{prop:hessian},
$$\Hess_M(f)(X,X)=\Hess_B(F)(X,X),$$
as desired. In vertical directions, by Proposition \ref{prop:sigma_restrict} and assumption (i) or the assumption $T=0$, 
$$\Hess_M(f)(Y,Y)=\Hess_\Sigma(f)(Y,Y)=0$$ 
because $f$ is invariant, as desired. To conclude, write
$$\Hess_M(f)(X,Y)=X(Y(f))-(\nabla_XY)^{\ver}(f)-(\nabla_XY)^{\hor}(f).$$
The first two terms vanish because $f$ is invariant. Under assumption (i) the third term also vanishes, as desired. Alternatively notice that, for every $X'$ horizontal, 
$$(\nabla_XY)^{\hor}\cdot X'=\nabla_XY\cdot X'=-Y\cdot \nabla_XX'=-Y\cdot (1/2) [X,X']^{\ver}.$$
It follows that the assumption $A=0$ implies the vanishing of $(\nabla_XY)^{\hor}$, thus of the third term, as desired.

Part 3: Assume $\Sigma_b$ is minimal. Choosing a local vertical ON basis we can write
\begin{equation*}
\Hess_M(f)(e_i,e_i)=\Delta_\Sigma f=0,
\end{equation*}
where we use the invariance of $f$. The conclusion follows by taking the trace of $\Hess$ wrt horizontal vectors.
\end{proof}

\begin{cor}\label{cor:potential}
Let $\pi:M\rightarrow B$ be a Riemannian submersion and $F:B\rightarrow \R$. 
\begin{enumerate}
\item If $T=0$ and $A=0$ on $M$ then $\pi^*F$ is convex iff $F$ is convex.
\item If all fibres are minimal then $\pi^*F$ is subharmonic iff $F$ is subharmonic.
\end{enumerate}
\end{cor}
The second statement goes back to \cite{EellsSampson}, \cite{Fuglede}.

\ 

\textit{Example. }Consider $M:=\R^2\setminus 0$, seen as a submersion over $B:=\R^+$ with fibres given by the circles of radius $r$. The radius function $r$ on $M$ is invariant, and it is convex: one can check this either via a direct calculation, or by looking at its graph: the upper half of the cone $z^2=x^2+y^2$. It descends to the linear function $r$ on $B$, but it is strictly convex in directions tangent to the fibres. This example confirms that, without additional assumptions on the fibres, we cannot expect that $\Hess_M(f)=\pi^*(\Hess_B(F))$.

\ 

\textit{Remark. }The above example suggests yet another interesting situation. Let $M$ be a Riemannian manifold and $\Sigma$ a compact submanifold. Let $r$ denote the distance function from $\Sigma$. Then $|\nabla r|\equiv 1$ so $r:M\setminus\Sigma\rightarrow B:=(0,\infty)$ is (where defined) a Riemannian submersion with codimension 1 fibres. For dimensional reasons, $A=0$. Choose $F:(0,\infty)\rightarrow\R$. Set $f:=\pi^*F$. Then for $X$ horizontal and $Y$ vertical, $\Hess_M(f)(X,X)=\Hess_B(F)(X,X)$ and $\Hess_M(f)(X,Y)=0$ as in Proposition \ref{p:potential}, Part 2. In vertical directions,
$$\Hess_M(f)(Y,Y)=\Hess_\Sigma(f)(Y,Y)-\nabla_YY^{\hor}(f)=-(\nabla_YY\cdot N)df(N),$$
by invariance of $f$ and setting $N:=\nabla r$. Its sign can thus be controlled via sign assumptions on the second fundamental form of the fibres and on the derivative $F'$. In this situation one typically finds that $\Hess_M(f)\neq\pi^*(\Hess_B(F))$.

Applied to the above example, this confirms the convexity of $f:=\pi^*(r^2)$.

\paragraph{Classical pluri-potential theory.}The situation described above is very simple. Using the invariance of $f$ we have built a perfect dictionary between analogous function theories on $M$ and on $B$. On the flip side, this means that there is basically no gain in passing between these two spaces.

The situation is more interesting in the case of complex manifolds: here, one can produce a dictionary between two different function theories. First, some definitions.

\begin{definition}Let $M$ be a complex manifold.

A function $f:M\rightarrow\R$ is pluri-subharmonic (PSH) iff its Levi form is non-negative, ie $\partial\bar\partial f(Z,\bar Z)\geq 0$ for all complexified vectors of the form $Z=X-iJX$.
\end{definition}
A direct calculation shows that 
\begin{equation}\label{eq:Levi_bis}
\partial\bar\partial f(Z,\bar Z)=2i\partial\bar\partial f (X,JX).
\end{equation}

Furthermore, using Cartan's formula one finds
$$\partial\bar\partial f(Z,\bar Z)=d\bar\partial f(Z,\bar Z)=\dots=X(Xf)+JX(JX f)-2(\Imm \partial f)[X,JX].$$
Now notice that 
$$\partial f(Y)=\partial f(\frac{Y-iJY}{2}+\frac{Y+iJY}{2})=\frac{1}{2}\partial f(Y-iJY)=\frac{1}{2}df(Y-iJY),$$ 
so $2(\Imm \partial f)(Y)=-df(JY)$. We may conclude that 
\begin{equation}\label{eq:Levi}
\partial\bar\partial f(Z,\bar Z)=X(Xf)+JX(JX f)-df(J[X,JX]).
\end{equation}
Assume for example that $\Sigma\hookrightarrow M$ is a complex curve. We may then choose $X,JX$ to correspond to complex coordinates on $\Sigma$, so that the Lie bracket vanishes. This proves, for example, the following well-known fact.

\begin{prop}
A function $f:M\rightarrow\R$ is PSH iff, restricted to any complex curve and choosing complex coordinates, it is locally subharmonic.
\end{prop}

This implies a maximum principle: if $f$ is PSH on $M$ and has a maximum point, then it is subharmonic on each complex line through that point, thus (by an open/closed argument) constant.

As before, this has strong consequences on submanifold geometry.

1. Any region where $M$ admits a strict PSH function $f$ cannot contain compact complex submanifolds. Geometrically, we can think of these as regions where a K\"ahler structure admits a potential function. Indeed, the restriction of $f$ to the submanifold would again be strictly PSH and this would contradict the maximum principle. Thus, for example, $f:=\sum z_i\bar z_i$ shows that there exist no compact complex submanifolds in $\C^n$.

2. Let $L$ be a positive holomorphic line bundle over $M$, ie the integral class $c_1(L)$ contains a strictly positive (1,1) form. $M$ is then projective, so the $\partial\bar\partial$-lemma allows us to build a Hermitian metric $h$ on $L$ with positive curvature. Assume given a holomorphic section $\sigma$, vanishing on $Z\subseteq M$. On $M\setminus Z$ let us define $H:=h(\sigma,\sigma)$. The curvature can then be written $i\partial\bar\partial(-\log H)$, so $-\log H$ is strictly PSH. It follows that no compact complex submanifold may be contained in $M\setminus Z$, ie the divisor of $\sigma$ intersects all such submanifolds (as implied by the Nakai positivity criterion).

\ 

We are interested in the following result, due to Lassalle \cite{Lassalle}. Given a compact Lie group $G$, we will let $G^c$ denote its complexification. The space $G^c/G$ then admits a natural notion of convexity: a function $F:G^c/G\rightarrow \R$ is convex iff it is convex along each real 1-parameter curve obtained as the projection of a complex 1-parameter subgroup of $G^c$.

\begin{thm}\label{thm:pluripotential} 
Consider the projection $\pi:G^c\rightarrow G^c/G$.
\begin{enumerate}
\item Assume $f=\pi^*F$. If $f$ is PSH then $F$ is convex. 

Furthermore, if $G$ is Abelian then $f$ is PSH iff $F$ is convex.
\item Choose any $f$ on $G^c$. Let $F$ denote the function on $G^c/G$ obtained, at each point, as the fibre-wise integral of $f$ wrt the Haar measure on $G$. If $f$ is PSH then $\pi^*F$ is PSH and invariant on $G^c$, so $F$ is convex.
\item Choose any $f$ on $G^c$. Let $F$ denote the fibre-wise supremum on $G^c/G$. If $f$ is PSH then $\pi^*F$ is PSH and invariant on $G^c$, so $F$ is convex.
\end{enumerate}
\end{thm}

\begin{proof}
Part 1: In equation \ref{eq:Levi}, assume $X$ is a left-invariant vector field on $G^c$. Then the Lie bracket term disappears because the Lie algebra is complex. If furthermore $X$ is vertical wrt the projection then $X^2f=0$, while $(JX)^2f$ coincides with the second derivative of $F$. This proves the first statement. 

The reason the ``iff" does not hold in general is that we have restricted our attention to the subspace of vertical vectors. A general $X$ is of the form $X=U+JV$. Substituting and simplifying we are left with the identity
$$\partial\bar\partial f(Z,\bar Z)=(JU)^2f+(JV)^2f+df([U,JV]-[V,JU]).$$
If $G$ is Abelian these Lie brackets vanish and we obtain the converse result.

Part 2: To prove the PSH condition it is convenient to adopt a slightly different viewpoint. Consider the function 
$$\tilde{f}:G^c\times G\rightarrow \R, \ \ \tilde{f}(h,g):=f(hg).$$
Since the $G$-action is holomorphic, each function $h\mapsto\tilde{f}(h,g)$ is PSH on $G^c$. Now integrate this family of functions wrt $G$. Differentiating under the integral (or, in the weak case, see \cite{Lelong} Thm 2.2.1) shows that this produces an invariant PSH function on $G^c$. By construction this function coincides with $\pi^*F$. Part 1 then shows that it corresponds to a convex function below.

Part 3: Consider the family of PSH functions $h\mapsto\tilde{f}(h,g)$ as above, and the corresponding function $h\mapsto\sup\{\tilde{f}(h,g):g\in G\}$. By construction, the latter coincides with the function $F$. It is a standard result in classical pluri-potential theory that its upper semi-continuous regularization $F^*$ is PSH. In our case $F$ is continuous so $F^*=F$. It is also invariant on $G^c$, so we can conclude as before.
\end{proof}

\paragraph{Conclusions.}Theorem \ref{thm:pluripotential} is interesting because the operators belong to very different realms: PSH above, convex below. Notice that (i) the statement makes sense only because the quotient space admits an intrinsic notion of convexity, (ii) convexity is a global property on $B$.

Our goal in the following sections will be to relate different potential theories using Riemannian submersions. Lassalle's result will be our model. To define convexity below we will use the Riemannian structure on $B$. Proposition \ref{p:potential}, Part 2, showcases two different types of assumptions. As already noted, if we focus on a critical point the Hessian does not actually depend on the metric; any convexity results at that point are perhaps better classified as stability results for the critical point. If instead we focus on geometric assumptions on a fibre, we can obtain global convexity statements simply by assuming that all fibres share that same property, as in Corollary \ref{cor:potential}.

\section{Generalized pluri-potential theories}\label{s:HL}

In the previous section we introduced two second-order operators on functions, defined using different geometric structures: $g$ and $J$. We will see below that when $M$ is K\"ahler, ie the two geometric structures are compatible, the two operators are closely related. 

It turns out that these operators are part of a much broader picture, concerning the Hessian and Grassmannians, developed in particular by Harvey and Lawson in many papers across the past 15 years, eg \cite{HL2}, \cite{HL3}. The basic idea is as follows.

\begin{definition}
Let $M$ be a Riemannian manifold. Let $\mathcal{G}$ denote a subset of the Grassmannian of $p$-planes in $TM$. 

A function $f:M\rightarrow\R$ is $\mathcal{G}$-PSH if $\tr_{|\pi}\Hess(f)\geq 0$, for all $\pi\in\mathcal{G}$.
\end{definition}

A few examples:
\begin{itemize}
\item Choose $\mathcal{G}$ to be the set of all $p$-planes in $TM$. One can prove that $f$ is $\mathcal{G}$-PSH iff the sum of the first $p$ eigenvalues of $\Hess(f)$ is non-negative.

When $p=1$ we obtain the usual notion of convexity. When $p=\dim(M)$ we obtain the usual  subharmonic functions.
\item Let $M$ be a K\"ahler manifold. Choose $\mathcal{G}$ to be the Grassmannian of complex lines. Using the fact that the Levi-Civita connection is torsion-free and the K\"ahler condition, we may write $J[X,JX]=J(\nabla_XJX-\nabla_{JX}X)=-\nabla_XX-\nabla_{JX}JX$, so
\begin{equation}\label{eq:kahlerPSH}
\partial\bar\partial f(Z,\bar Z)=\Hess(f)(X,X)+\Hess(f)(JX,JX).
\end{equation}
This shows that $f$ is PSH in the complex-analytic sense iff it is $\mathcal{G}$-PSH.
\item Let $\alpha$ be a calibrating $p$-form \cite{HL1} on $M$, ie: for all oriented $p$-planes $\pi$, $\alpha_{|\pi}\leq \vol(\pi)$. Choose $\mathcal{G}$ to be the set of calibrated $p$-planes, ie $\alpha_{|\pi}= \vol(\pi)$.
\end{itemize}
In each case the notion of $\mathcal{G}$-PSH functions produces a function theory on $M$. Under fairly general assumptions, Harvey-Lawson prove that all standard analytic results, such as the max principle, continue to hold. As seen, in some cases one recovers classical function theories. In other cases the function theory is new. 

The relevance to geometry is that, in each case, $\mathcal{G}$-PSH functions control, via the max principle and in the sense seen above, the class of minimal submanifolds defined by the condition that the tangent spaces lie in $\mathcal{G}$.

\ 

We are particularly interested in the calibrated case. We choose to focus on the following two instances.

\paragraph{K\"ahler manifolds.}Let $M$ be a K\"ahler manifold, with complex structure $J$ and K\"ahler 2-form $\omega$. The complex structure defines the PSH condition; the corresponding function theory is then closely related to complex submanifolds, defined by the condition $J(T\Sigma)=T\Sigma$. Alternatively, the K\"ahler form is a calibrating form. The calibrated planes are the complex lines. As seen, the two function theories coincide.

An important aspect of this geometry is that we can use $\omega$ to also define the class of Lagrangian submanifolds $\Sigma^n\hookrightarrow M^{2n}$, via the condition $\omega_{|T\Sigma}\equiv 0$ or, equivalently, $J(T\Sigma)\perp T\Sigma$: this second point of view emphasizes the fact that complex and Lagrangian submanifolds are, roughly speaking, defined by opposite conditions.

The contrast between these two classes of submanifolds will be a key element in the following sections. 

\paragraph{$G_2$ manifolds.}Assume $M$ has dimension 7. A $G_2$ structure on $M$ is a 3-form $\phi$ pointwise isomorphic to the standard 3-form
$$123+1(45+67)+2(46-57)-3(47+56)$$
on $\R^7$, where we use the short-hand notation $123=e^1\wedge e^2\wedge e^3$, etc.

General $G_2$ theory, see eg \cite{G2lectures}, shows that $\phi$ defines on $M$ a Riemannian metric and an orientation. The dual 4-form $\psi=\star\phi$ is pointwise isomorphic to
$$4567+23(45+67)-13(46-57)-12(47+56).$$
Both $\phi$ and $\psi$ are calibrating forms. The corresponding calibrated 3-planes (known as associative, modelled on $\mbox{span}\{e_1,e_2,e_3\}$) and 4-planes (known as coassociative, modelled on $\mbox{span}\{e_4,e_5,e_6,e_7\}$), are orthogonal. Both define a corresponding class of $\mathcal{G}$-PSH functions. The fact that there exist such function theories, with good analytic properties, specific to $G_2$ manifolds is a surprising consequence of Harvey-Lawson's work. 

We shall focus on the following choice.

\begin{definition}
Let $\mathcal{G}$ be the set of associative 3-planes $\pi$, defined by the condition $\phi_{|\pi}=\vol(\pi)$.
The class of $\phi$-PSH functions on $M$ is the class of $\mathcal{G}$-PSH functions determined by this set. 
\end{definition}

The calibrated submanifolds corresponding to $\phi$, ie $\phi_{|T\Sigma}=\vol_{\Sigma}$, are known as associative submanifolds. They are 3-dimensional. The calibrated submanifolds corresponding to $\psi$, ie $\psi_{|T\Sigma}=\vol_{\Sigma}$, are known as coassociative submanifolds. They are 4-dimensional. One can prove that, up to orientation, $\Sigma$ is coassociative iff $\phi_{|T\Sigma}\equiv 0$. 

Notice: the calibrated condition means that associative submanifolds play a role analogous to complex submanifolds, while the condition $\phi_{|T\Sigma}\equiv 0$ makes coassociative submanifolds analogous to Lagrangian submanifolds.

\ 

To conclude, recall that the general theory \cite{HL1} shows that, if a calibrating form is closed ie $d\alpha=0$, then calibrated submanifolds are automatically volume-minimizing. It follows that they are controlled by the corresponding pluri-potential theory.

The closed condition $d\omega=0$ is part of the definition of K\"ahler manifolds. This implies not only that complex curves (more generally, complex submanifolds) are volume-minimizing, but also provides a sufficient condition for the (local) existence of Lagrangian submanifolds.  In the $G_2$ case we will usually require $d\phi=0$. This implies that associative submanifolds are volume minimizing and provides a sufficient condition for the (local) existence of coassociative submanifolds.

\ 

\textit{Remark. }Let $\Sigma$ be a compact submanifold of $M$. Let $r$ denote the distance function from $\Sigma$, defined on a $C^0$-neighbourhood of $\Sigma$. Understanding the properties of $r$ generally requires a careful study of geodesics and of the variation formulae for length. Under various assumptions, eg $\Sigma=\{p\}$ or $\Sigma$ totally geodesic and $M$ a Hadamard manifold, $r^2$ is known to be convex. Various other constructions of geometrically interesting convex functions on $M$ are known \cite{Shiga}. 

Harvey-Wells provide local examples of strict PSH functions in $\C^n$. Recall that a submanifold $\Sigma^n$ in $\C^n$ is totally real iff $T\Sigma$ contains no complex lines, ie $J(T\Sigma)\cap T\Sigma=\{0\}$. 
If $\Sigma$ is totally real then \cite{HarveyWells} $r^2$ is PSH.

Harvey-Lawson have extended this result to other classes of $\mathcal{G}$-PSH functions. It follows, for example, that given any compact coassociative submanifold there exists a $C^0$-neighbourhood containing no compact associative submanifolds.

Global existence results are of a different nature. In the complex case this is equivalent to the Stein condition.

\ 

\textit{Remark. }In many cases, $\mathcal{G}$-PSH functions can alternatively be defined via a positivity condition on an appropriate tensor.

The most obvious example is convexity, defined via $\Hess(f)\geq 0$. In the K\"ahler case, $f$ is PSH iff the 2-form $i\partial\bar\partial(f)$ has the property  $i\partial\bar\partial(f)(X,JX)\geq 0$, for all $X$. In the $G_2$ case, if $d\phi=0$ then $f$ is $\mathcal{G}$-PSH iff the 3-form 
$$\mathcal{H}^\phi(f):=d(\nabla f\lrcorner\phi)-\nabla_{\nabla f}\phi$$ 
has the following property: $\mathcal{H}^\phi(f)_{|\pi}\geq 0$, for all calibrated 3-planes $\pi$.

\ 

\section{Invariant functions}\label{s:invariant}

We shall investigate potential theories on manifolds $M$, $B$ related by a Riemannian submersion. For the moment we shall concentrate on invariant functions. Proposition \ref{p:potential} covers the general Riemannian case, so we shall focus here on  K\"ahler and $G_2$ manifolds.

\paragraph{K\"ahler manifolds.}It is clear from the definitions that
$$\mbox{Convex}\Rightarrow\mbox{PSH}\Rightarrow\mbox{Subharmonic}.$$
 
Our model Theorem \ref{thm:pluripotential}, Part 1, suggests however the possibility of a converse implication: $\mbox{PSH}\Rightarrow\mbox{Convex}$. In that theorem the fibres are modelled by $G$; the vertical spaces are modelled by the Lie algebra $\mathfrak{g}$. The fibres are thus totally real in $G^c$. The Abelian assumption implies that the natural transverse subspace, $J\mathfrak{g}$, is integrable. In the K\"ahler setting Lagrangian submanifolds are the natural analogue of totally real submanifolds. Proposition \ref{p:potential} has already highlighted the importance of an integrability condition. The following result thus seems to be the natural link between Proposition \ref{p:potential} and Theorem \ref{thm:pluripotential}.

\begin{prop}\label{p:kahler}
Let $M$ be a K\"ahler manifold and $\pi:M\rightarrow B$ a Riemannian submersion.  
Choose an invariant function $f=\pi^*F$ and a fibre $\Sigma_b$.
\begin{enumerate}
\item Assume $\Sigma_b$ is Lagrangian. If, at each point of $\Sigma_b$, either 
\begin{itemize}
\item[(i)] $\Sigma_b$ is contained in the critical locus of $f$, or
\item[(ii)] $T=0$ and $A=0$
\end{itemize}
then $f$ is PSH iff $f$ is convex.

In particular, $F$ is convex at $b\in B$ iff $f$ is PSH at each point of $\Sigma_b$.
\item Assume $\Sigma_b$ is minimal. If $f$ is PSH at each point of $\Sigma_b$, then $F$ is subharmonic at $b$.
\end{enumerate}
\end{prop}
\begin{proof}
Part 1: We already know that $f$ convex implies $f$ PSH.

Conversely, assume $f$ is PSH. We have seen in Proposition \ref{p:potential} that 
$\Hess_M(f)=\pi^*(\Hess_B(F))$. In particular, given $X$ horizontal and $Y$ vertical, 
$$\Hess_M(f)(Y,Y)=0, \ \ \Hess_M(f)(X,Y)=0.$$ 
The Lagrangian condition implies that $JX$ is vertical so, using Equation (\ref{eq:kahlerPSH}),
$$\Hess_M(f)(X,X)=\Hess_M(f)(X,X)+\Hess_M(f)(JX,JX)\geq 0.$$
It follows that $f$ is convex. 

Part 2: PSH implies subharmonic, so we can apply Proposition \ref{p:potential}.
\end{proof}

\textit{Remark. }Consider the function $f(x,y)=2x^2-y^2$. It is PSH and constant on the Lagrangian curves defined by the hyperbolae $2x^2-y^2=c$, but not convex. This shows the need for additional assumptions, as in the proposition.

Notice that Part 2 does not require the Lagrangian condition. If $\Sigma$ is Lagrangian and $\{e_j\}$ is an ON basis for $T_p\Sigma$, then $\Delta_Mf=2i\partial\bar\partial f(e_j,Je_j)$. This generalizes the usual formula 
$\Delta f\,dxdy=2i\partial\bar\partial f$ in $\C$. 

\ 

Proposition \ref{p:kahler} has the following application.

\begin{cor}\label{cor:unique}
Let $M$ be a K\"ahler manifold and $\pi:M\rightarrow B$ a Riemannian submersion with Lagrangian fibres. Let $f=\pi^*F$ be strictly PSH. Then every critical point of $F$ is a local minimum.
\end{cor}

Notice: appropriate conditions on $B$, together with Morse theory or min-max arguments, would ensure that $F$ has at most one such point. 

\ 

\textit{Remark. }It is clear that strict convexity implies strict PSH. The above proof shows that the converse is not necessarily true, eg $f(x,y)=x^2$. 

\paragraph{$G_2$ manifolds.}The situation for $G_2$ manifolds is closely analogous. It is clear from the definitions that convexity implies $\phi$-PSH. We have already mentioned that coassociative submanifolds are the analogue of Lagrangian submanifolds. The following result confirms this. 

\begin{prop}\label{p:G2}
Let $M$ be a manifold with a closed $G_2$ structure and $\pi:M\rightarrow B$ a Riemannian submersion.  
Choose an invariant function $f=\pi^*F$ and a coassociative fibre $\Sigma_b$.
\begin{enumerate}
\item If, at each point of $\Sigma_b$, either
\begin{itemize}
\item[(i)] $\Sigma_b$ is contained in the critical locus of $f$, or 
\item[(ii)] $T=0$ and $A=0$
\end{itemize}
then the conditions: $f$ $\phi$-PSH, $f$ convex, $F$ convex (at $b\in B$) are equivalent.

In particular, if all fibres are coassociative and $f$ is strictly $\phi$-PSH then every critical point of $F$ is a local minimum.
\item If $f$ is $\phi$-PSH at each point of $\Sigma_b$ then $F$ is subharmonic at $b$.
\end{enumerate}
\end{prop}
\begin{proof}
General $G_2$ theory shows that, given any unit horizontal $X\in (T\Sigma_b)^\perp$, there exist unit vertical vectors $u,v\in T\Sigma_b$ such that $\pi:=<X,u,v>$ is associative. For example, in the pointwise model we may use the transitivity of the $G_2$ action to assume $X=1$; we can then choose $u=4$, $v=5$. The assumptions imply that $\Hess_M(f)=0$ in the directions $u,v$. It follows that
$$\Hess_M(f)(X,X)=\Hess_M(f)(X,X)+\Hess_M(f)(u,u)+\Hess_M(f)(v,v).$$ 
The proof continues as for Proposition \ref{p:kahler}, Part 1.

For the second part it suffices to choose the associative plane $\pi:=(T\Sigma_b)^\perp$.
\end{proof}
Compared to Proposition \ref{p:kahler}, notice that in Part 2 we can replace minimality with the assumption that $\Sigma_b$ be coassociative.

\section{Second variation formulae}\label{s:secondvar}

Let $\Sigma^k$ be a compact oriented manifold. Let $(M^n,g)$ be a Riemannian manifold. Let $\iota_t:\Sigma\hookrightarrow M$ be a curve of immersions. The induced metric defines volume forms $\vol_t$ on the image submanifolds $\Sigma_t$. We may pull them back to $\Sigma$, obtaining the curve of volume forms $\iota_t^*\vol_t$. This allows us to compute the volume of the image submanifolds using the fixed manifold $\Sigma$: $\Vol(\Sigma_t)=\int_\Sigma \iota_t^*\vol_t$. To simplify the notation, we will often hide the role of $\iota_t$.

Set $Z:=\frac{d}{dt}\iota_t$. Wrt the submanifolds, we can write $Z$ as a sum of tangential and normal components: $Z=Z^T+Z^\perp$. As usual, $H=\tr_\Sigma(\nabla^\perp)$ will denote the mean curvature vector field.

We are interested in understanding how the volume forms and the volume functional change with $t$, both in the general setting and in the specific calibrated cases of interest to us.

\paragraph{The classical setting.}It is well-known that, in the most general Riemannian setting, the volume forms change according to the following formulae:
\begin{align}\label{eq:volvariation}
\frac{d}{dt}\vol_{t|t=0}&=(\div_{\Sigma}(Z^T)-H\cdot Z^\perp)\vol_0,\\
\label{eq:volsecvariation}
\frac{d^2}{dt^2}\vol_{t|t=0}&=(-(\nabla_{e_i}e_j,Z)^2-(R(e_i,Z)Z,e_i)+\mbox{div}_\Sigma((\nabla_Z Z)^T)\nonumber\\
&\ \ \ -(H,(\nabla_Z Z)^\perp)+\|(\nabla_{e_i}Z)^\perp\|^2+(H,Z)^2)\vol_0,
\end{align}
where $\div_\Sigma$ and $H$ are calculated wrt $\Sigma_0$ and $\nabla$, $R$ are the connection and curvature tensor of $M$. The first variation formula is straight-forward:
\begin{equation*}
\frac{d}{dt}\Vol(\Sigma_t)_{|t=0}=-\int_\Sigma (H\cdot Z^\perp)\vol_0.
\end{equation*}
In order to simplify the second variation formula, we shall assume $Z$ is always perpendicular to $\Sigma_t$: this will not change the volumes. Let us also assume $\Sigma_0$ is minimal: $H=0$. We then find \cite{Simons}:
$$\frac{d^2}{dt^2}\Vol(\Sigma_t)_{|t=0}=\int_\Sigma(-(\nabla_{e_i}Z\cdot e_j)^2-R(e_i,Z)Z\cdot e_i+(\nabla_{e_i}Z\cdot f_j)^2)\vol_0,$$
where $e_1,\dots,e_k$ is a orthonormal basis of $T_p\Sigma$ at any given point, $f_1,\dots,f_{n-k}$ is a orthonormal basis of $T_p\Sigma^\perp$.

\paragraph{K\"ahler manifolds.}It is an interesting fact that, when $\Sigma_0$ is minimal Lagrangian, the terms in the second variation formula can be combined in a different way to produce the following result, which is much stronger than its Riemannian analogue. 

To simplify, we will again restrict to normal variations $Z$ and use the isomorphism, valid for Lagrangian submanifolds,
\begin{equation*}
T\Sigma^\perp\simeq\Lambda^1(\Sigma), \ \ Z\mapsto\zeta:=\omega(Z,\cdot)_{|T\Sigma}.
\end{equation*}
Then \cite{Oh} (see also \cite{PaciniRaffero})
\begin{align}\label{eq:kahlervolsecvariation}
\frac{d^2}{dt^2}\vol_{t|t=0}&=(1/2)\Delta_\Sigma \|Z\|^2+(\Delta_\Sigma\zeta,\zeta)-\Ric(Z,Z)+\div_\Sigma((\nabla_ZZ)^T),\\
\frac{d^2}{dt^2}\Vol(\Sigma_t)_{|t=0}&=\int_\Sigma ((\Delta_\Sigma \zeta,\zeta)-\Ric(Z,Z))\vol_0\nonumber,
\end{align}
where $\Ric$ is the Ricci curvature of $M$ and $\Delta_\Sigma$ is the Hodge Laplacian on $\Sigma_0$.

\paragraph{$G_2$ manifolds.}Let $M$ be endowed with a closed $G_2$ structure $\phi$. This ensures that any initial compact coassociative submanifold generates a smooth moduli space $\mathcal{M}$ of coassociative deformations. The dimension of this space depends on the topology of $\Sigma$: specifically, it coincides with $b^2_+(\Sigma)$. 

Assume $\Sigma_0$ is minimal coassociative. As in the K\"ahler case, the second variation formula takes a specific form in this context. To simplify, we shall restrict to variations through coassociative submanifolds, ie tangent to $\mathcal{M}$. This will suffice for our purposes, below. We shall again restrict to normal variations $Z$. Then \cite{PaciniRaffero}
\begin{align}\label{eq:G2volsecvariation}
\frac{d^2}{dt^2}\vol_{t|t=0}&=d(\iota_Z\tau_2\wedge\iota_Z\phi)_{|\Sigma}+(\tau_2\wedge\gamma_Z)_{|\Sigma}+(\iota_Zd\tau_2\wedge\iota_Z\phi)_{|\Sigma},\\
\frac{d^2}{dt^2}\Vol(\Sigma_t)_{|t=0}&=\int_\Sigma\tau_2\wedge\gamma_Z+\iota_Zd\tau_2\wedge\iota_Z\phi\nonumber,
\end{align}
where $\tau_2\in\Lambda^2_{14}(M)$ is the torsion form of $(M,\phi)$ and $\gamma_Z\in\Lambda^2_-(\Sigma_0)$ is defined in terms of the second fundamental form: 
$$\gamma_Z(X_1,X_2):=\iota_Z\phi((\nabla_{X_1} Z)^T,X_2)+\iota_Z\phi(X_1,(\nabla_{X_2} Z)^T).$$
It is shown in \cite{Bryant}, see also \cite{PaciniRaffero}, that $d\tau_2=1/2(\star(\tau_2\wedge\tau_2)-\i(\Ric))$, for a certain map $\i:\Sym^2\rightarrow\Lambda^3$: this indicates that the Ricci tensor plays a role in the coassociative second variation formula, similar to its role in the Lagrangian case. Stability of $\Sigma_0$ is clearly related to the positivity of the bilinear forms
$$Q_\pi:\pi^\perp\times \pi^\perp\rightarrow \Lambda^4(\pi),\ \ Q_\pi(Z_1,Z_2):=(\iota_{Z_1}d\tau_2\wedge\iota_{Z_2}\phi)_{|\pi},$$
for all $\pi=T_p\Sigma_0$. One can show \cite{PaciniRaffero} that $Q_\pi$ are symmetric.

\begin{definition}
We will say that a manifold $M$ endowed with a closed $G_2$ structure satisfies the positivity condition iff $Q_\pi\geq 0$, for all $\pi$ in the Grassmannian of coassociative 4-planes.
\end{definition}

For example, \cite{PaciniRaffero} shows that the class of $G_2$ structures defined by a certain ``quadratic condition" \cite{Bryant}, depending on a parameter $\lambda$, satisfies the positivity condition when $|\lambda|\leq 1/\sqrt{21}$. In particular, this hold for the class of ``Extremally Ricci-Pinched" $G_2$ manifolds \cite{Bryant}, corresponding to the case $\lambda=1/6$.

\paragraph{Application.}We have seen how to use the distance function to produce geometrically meaningful convex/PSH functions near certain submanifolds. We can use the above results to obtain a second class of examples. 

Assume $\pi:M\rightarrow B$ is a Riemannian submersion with compact oriented fibres. We can then restrict the volume functional to the submanifolds defined by the fibres, obtaining a function on $B$, thus an invariant function on $M$. The second variation integral formulae seen above, together with Propositions \ref{p:potential}, \ref{p:kahler}, \ref{p:G2}, then lead to the following conclusions.

\begin{cor}\label{cor:volume}
In the setting described above:
\begin{enumerate}
\item If $M$ is Riemannian with non-positive sectional curvature, then the volume functional is convex along any totally geodesic fibre.
\item If $M$ is K\"ahler with non-positive Ricci curvature, then the volume functional is convex (equivalently, PSH) along any minimal Lagrangian fibre.
\item If $M$ has a closed $G_2$ structure which satisfies the positivity condition and the fibres are coassociative, then the volume functional is convex (equivalently, PSH) along any totally geodesic fibre.
\end{enumerate}
\end{cor}

\textit{Remark. }The differences between Parts 2, 3 correspond to the differences in the second variation formulae. In the $G_2$ case we need all fibres to be coassociative because the formula applies only to coassociative variations, and we need stronger assumptions to control the sign of the integrand.

\section{Integral and supremum functions}\label{s:integralsup}

Again following the tracks of Theorem \ref{thm:pluripotential}, we now turn to the case of integral and supremum functions. We shall assume all fibres are compact and oriented.

\paragraph{Riemannian manifolds.}

\begin{thm}\label{thm:potential_bis}
Let $\pi:M\rightarrow B$ be a Riemannian submersion. Choose a function $f:M\rightarrow \R$. For any $b\in B$, let $F(b):=\int_{\Sigma_b}f\vol$ denote the fibre-wise integral of $f$.
\begin{enumerate}
\item Assume $M$ has non-positive sectional curvature and a fibre $\Sigma_b$ is totally geodesic. If $f$ is non-negative and convex (or subharmonic) at each point of $\Sigma_b$, then $F$ is convex (or subharmonic) at $b$.
\item Assume all fibres have constant volume. Then $f$ convex implies $F$ convex. 
\item Assume all fibres are minimal. Then $f$ subharmonic implies $F$ subharmonic.
\end{enumerate}
\end{thm}
\begin{proof}
Part 1: Choose $X\in T_bB$ and a geodesic $\alpha$ in $B$ through $b$ and with direction $X$, so that $\Hess_B(F)(X,X)=\frac{d^2}{dt^2}(F\circ\alpha)(t)_{|t=0}$. In order to perform this calculation we shall lift the geodesic to all points of the fibre $\Sigma_b$. This produces a variation $\iota_t:\Sigma\hookrightarrow M$ of $\Sigma_b$ through fibres $\Sigma_{\alpha(t)}$, where $\Sigma$ denotes the abstract fibre. The infinitesimal variation is the horizontal lift of $X$, so we may write $Z=X$. Then
\begin{align*}
\frac{d^2}{dt^2}(F\circ\alpha)(t)_{|t=0}&=\int_{\Sigma}\frac{d^2}{dt^2}(f\vol_t)_{|t=0}\\
&=\int_\Sigma\frac{d^2f}{dt^2}_{|t=0}\vol_0+2\frac{df}{dt}_{|t=0}\frac{d}{dt}\vol_{t|t=0}+f\frac{d^2}{dt^2}\vol_{t|t=0}.
\end{align*}
Let us look at the terms in the integrand. The first is $\Hess_M(f)(X,X)$. In the second term we can use equation \ref{eq:volvariation}: the divergence term vanishes because the variation is normal, the other term vanishes because $H=0$. In the third term we can use equation \ref{eq:volsecvariation}. The terms containing $H$ and the second fundamental form vanish by assumption on $\Sigma_b$ while the term $\mbox{div}_\Sigma((\nabla_XX)^T)$ vanishes because the variation is through geodesics, so $\nabla_XX=0$. We thus obtain the identity
$$\Hess_B(F)(X,X)=\int_{\Sigma_b}\Hess_M(f)(X,X)+f(-g(R(e_i,X)X,e_i)+\|(\nabla X)^\perp\|^2)\vol_0.$$
The claim concerning convexity follows.  

The claim concerning subharmonicity is similar but uses the extra fact that, as in Proposition \ref{prop:sigma_restrict}, $\tr_{|T\Sigma_b}\Hess_M(f)=\Delta_{\Sigma_b} f$ and that $\int_{\Sigma_b}\Delta_{\Sigma_b} f\vol_0=0$.
Now let $X_j$ denote an orthonormal frame at $b$. Then, as before,
\begin{align*}
\Delta_BF&=\int_{\Sigma_b}\Hess_M(f)(X_j,X_j)+f(-g(R(e_i,X_j)X_j,e_i)+\|(\nabla X_j)^\perp\|^2)\vol_0\\
&=\int_{\Sigma_b}\Delta_Mf+f(-g(R(e_i,X_j)X_j,e_i)+\|(\nabla X_j)^\perp\|^2)\vol_0\geq 0.
\end{align*}
Part 2: The assumption implies that all expressions of the form $(H,Z)$, ie the horizontal component of $H$, vanish for all $t$. In this case, rather than using equation \ref{eq:volsecvariation}, it is better to notice that 
\begin{align*}
\frac{d^2}{dt^2}\vol_{t|t=0}&=-\frac{d}{dt}((H\cdot Z)\iota_t^*\vol_t)_{|t=0}\\
&=-\frac{d}{dt}((H,Z)_{|t=0}+(H,Z)^2)\iota_0^*\vol_0=0.
\end{align*}
As above, this leads to the equality
$$\Hess_B(F)(X,X)=\int_\Sigma\Hess_M(f)(X,X)\vol_0.$$
We may conclude as before. 

Part 3: The assumption implies that the volume is constant and it allows us to apply the argument of Parts 1,2 to the above formula for $\Hess_B(F)$.
\end{proof}

\textit{Remark. }If, as in Theorem \ref{thm:potential_bis}, $f$ is convex and $\Sigma_b$ is totally geodesic (or even just minimal), then the restriction to $\Sigma_b$ is subharmonic. The max principle then implies that $f$ is constant on $\Sigma_b$.

\ 

\textit{Remark. }One might alternatively be interested in the normalized integral function. This coincides (up to constants) with our integral function whenever the volume functional is constant, as in Parts 2, 3  above (or in the case of group actions, see below). In general, however, the normalization adds extra terms to the calculations above, making convexity more difficult to achieve.

\ 

\begin{cor}
Assume $M$ has a closed calibration and $\pi:M\rightarrow B$ is a Riemannian fibration with calibrated fibres. Then the fibres are minimal so we are in the situation of Theorem \ref{thm:potential_bis}, Parts 2, 3: the fibre-wise integral of any convex/subharmonic function on $M$ is a convex/subharmonic function on $B$.
\end{cor}

The corollary applies, for example, to holomorphic fibrations of K\"ahler manifolds, special Lagrangian fibrations of Calabi-Yau manifolds, or associative (coassociative) fibrations of $G_2$ manifolds such that $d\phi=0$ ($d\psi=0$).

\paragraph{K\"ahler manifolds.}In this setting, Oh's second variation formula provides positivity via an assumption on the Ricci curvature rather than on the sectional curvature; on the other hand, it also requires an integration by parts. In the context of integral functions, the simplest way to obtain a useful conclusion is by assuming that $f$ be constant along the specific fibre. As mentioned following Theorem \ref{thm:potential_bis}, this is sometimes a consequence of other seemingly weaker assumptions.

\begin{thm}\label{thm:potential_ter}
Let $M$ be a K\"ahler manifold and $\pi:M\rightarrow B$ a Riemannian submersion. Choose a function $f:M\rightarrow \R$. For any $b\in B$, let $F(b):=\int_{\Sigma_b}f\vol$ denote the fibre-wise integral of $f$.

Assume $M$ has non-positive Ricci curvature, a fibre $\Sigma_b$ is minimal Lagrangian, and $f$ is constant and non-negative along $\Sigma_b$. 

If $f$ is convex (or subharmonic) at each point of $\Sigma_b$ then $F$ is convex (or subharmonic) at $b$.
\end{thm}
\begin{proof}
Choose a direction $X\in T_bB$. Its horizontal lift has constant norm along $\Sigma_b$ so, using equation \ref{eq:kahlervolsecvariation} as in Theorem \ref{thm:potential_bis}, we obtain that, at $b\in B$,
$$\Hess_B(F)(X,X)=\int_{\Sigma_b}\Hess_M(f)(X,X)+f((\Delta\xi,\xi)-\Ric(X,X))\vol_0.$$ 
The rest of the proof is similar to that of Theorem \ref{thm:potential_bis}.

Alternatively, the special case where $f=\pi^*F$ can be seen as a consequence of Corollary \ref{cor:volume}.
\end{proof}

This result could, for example, be applied in the following context.

\ 

\textit{Example. }Assume that $\pi:M\rightarrow B$ is as above and that $f:M\rightarrow\C$ is holomorphic. Then $\log|f|$ is PSH, so $|f|$ is non-negative and PSH. Under appropriate assumptions, Proposition \ref{p:kahler} and Theorem \ref{thm:potential_ter} lead to convexity properties for the fibre-wise $L^1$-norm of $f$.

\paragraph{$G_2$ manifolds.} The $G_2$ case is similar.

\begin{thm}\label{thm:G2_bis}
Let $M$ be a manifold with a closed $G_2$ structure and $\pi:M\rightarrow B$ a Riemannian submersion with coassociative fibres.  Choose a function $f:M\rightarrow \R$. For any $b\in B$, let $F(b):=\int_{\Sigma_b}f\vol$ denote the fibre-wise integral of $f$.

Assume $M$ satisfies the positivity condition, a fibre $\Sigma_b$ is totally geodesic, and $f$ is constant and non-negative along $\Sigma_b$.

If $f$ is convex (or subharmonic) at each point of $\Sigma_b$ then $F$ is convex (or subharmonic) at $b$.
\end{thm}
\begin{proof}
Using equation \ref{eq:G2volsecvariation} as in Theorem \ref{thm:potential_bis}, we obtain that, at $b\in B$,
\begin{align*}
\Hess_B(F)(X,X)&=\int_{\Sigma_b}\Hess_M(f)(X,X)\vol_0+f\int_{\Sigma_b}d(\iota_Z\tau_2\wedge\iota_Z\phi)+\iota_Zd\tau_2\wedge\iota_Z\phi\\
&= \int_{\Sigma_b}\Hess_M(f)(X,X)\vol_0+f\int_{\Sigma_b}\iota_Zd\tau_2\wedge\iota_Z\phi.
\end{align*}
The rest of the proof is similar to that of Theorem \ref{thm:potential_bis}.
\end{proof}

\paragraph{Group actions.}Assume a compact group $G$ acts freely by isometries on $(M,g)$. The corresponding map $\pi:M\rightarrow B$ onto the orbit space is then automatically a Riemannian submersion. Choose a function $f:M\rightarrow\R$. Then:
\begin{itemize} 
\item $f$ is $\pi$-invariant iff it is $G$-invariant. 
\item Rather than using the induced Riemannian volume form to define the fibre-wise integral function $F$ on $B$, in this context it is natural to use the Haar measure $d\mu$. We thus set 
$$F(b):=\int_Gf_{|\mathcal{O}_b} d\mu,$$ 
where $\mathcal{O}_b$ denotes the $G$-orbit above $b\in B$. 
\end{itemize}
Notice two features of this integration process: (i) applied to an invariant function $f=\pi^*F$, integration reproduces that same function $F$, (ii) derivatives of the integral function do not require second variation formulae because the Haar measure does not depend on the specific point $b$. This eliminates the need for some of the above assumptions. 

\begin{thm}\label{thm:kahlerG}
Let $M$ be a Riemannian manifold endowed with a free isometric action of a compact Lie group $G$. Let $\pi:M\rightarrow B$ be the corresponding Riemannian submersion. 

Choose a function $f:M\rightarrow\R$. Let $F$ denote either the fibre-wise integral function, defined using the Haar measure, or the fibre-wise supremum function.
\begin{enumerate}
\item If $f$ is convex then $\pi^*F$ is convex and invariant.
\item Assume $M$ is K\"ahler and the action also preserves $J$. If $f$ is PSH then $\pi^*F$ is PSH and invariant.
\item Assume $M$ has a closed $G_2$ structure and the action preserves $\phi$. If $f$ is PSH then $\pi^*F$ is PSH and invariant.
\end{enumerate}
\end{thm}

The proof is similar to that of Theorem \ref{thm:pluripotential}, using also the results of \cite{HL3}.

We can now apply Propositions \ref{p:potential}, \ref{p:kahler}, \ref{p:G2} to the invariant functions $\pi^*F$ on $M$ so as to prove potential-theoretic properties of the functions $F$ on $B$.

\section{Interlude: Geometric fibrations}\label{s:fibrations}

We have seen above the parallel roles of Lagrangian and coassociative submanifolds. The existence of fibrations with such fibres has strong consequences. This topic acquired special prominence starting with \cite{SYZ}, \cite{GYZ}, which related calibrated fibrations to Mirror Symmetry. Let us review the two cases in turn.

\paragraph{Lagrangian fibrations.}Let $M$ be a symplectic manifold. Assume $\pi:M\rightarrow B$ is a fibration with smooth Lagrangian fibres. Choose local coordinates on $B$. Each component of the projection map is constant on the fibres and defines a Hamiltonian vector field which, by the Lagrangian assumption, is tangent to the fibres. The $n$ components define Poisson-commuting Hamiltonian functions, thus a completely integrable Hamiltonian system. The flows of their vector fields define a local $\R^n$-action. If the fibres are compact and we choose the local coordinates appropriately we may assume all flows are $2\pi$-periodic so the action descends to a local action of the standard torus. This action is free so each fibre is a torus. Such coordinates, thus the torus action, are unique up to the action of $\GL(n,\Z)$ (and translations): this corresponds to linear combinations of the component functions, thus of the Hamiltonian vector fields. One thus obtains an integral affine structure on $B$. We refer to \cite{Evans} for details.

\ 

\textit{Remark. }Basically, the above shows that $(M,\pi)$ is locally symplectomorphic to the standard $(\T^n\times\R^n,\pi)$. If $M$ is K\"ahler, this identification does not necessarily bring $J$ or $g$ to standard form: there are obvious obstructions in terms of curvature and of integrability of the horizontal distribution. In general we should thus not expect the local torus action to be isometric nor holomorphic.

\paragraph{Coassociative fibrations.}Let $M$ be a manifold with a closed $G_2$ structure. Assume $\pi:M\rightarrow B$ is a fibration with smooth coassociative fibres: see eg \cite{Baraglia},\cite{Donaldson} for general theory, \cite{PaciniRaffero} for a summary of examples with compact fibres due to Bryant, Fernandez and Lauret, \cite{PaciniRaffero} and \cite{JasonSpiro} for examples with non-compact fibres. Here, we shall assume the fibres are compact.

The projection provides an isomorphism between each $(T_p\Sigma_b)^\perp$ and $T_bB$ so the normal bundle, thus $\Lambda^2_+(\Sigma_b)$, is parallelizable. Choose a basis $Z_1,Z_2,Z_3$ for $T_bB$. We will use the same notation to denote the corresponding normal vector fields along $\Sigma_b$. The forms $Z_i\lrcorner\phi$ provide a basis for the self-dual forms at each point. Since each $Z_i$ corresponds to an infinitesimal deformation through coassociative fibres, deformation theory implies that each $Z_i\lrcorner\phi$ is also harmonic. We may think of $B$ as a submanifold in the moduli space of coassociative deformations. They locally coincide iff $b^2_+(\Sigma)=\dim(B)=3$.

\ 

If $\pi$ defines a Riemannian submersion, \cite{Baraglia} shows that the fibres, with the induced metric, are hyper-K\"ahler thus Ricci-flat (the argument requires only $d\phi=0$). The fibres must then be either tori or K3 surfaces. If furthermore the submersion is generated by the free action of a group $G$ which preserves $\phi$, then $G$ induces hyper-K\"ahler automorphisms of the fibres. Since K3 surfaces only have finite automorphism groups, the fibres must be tori.

\ 

\textit{Remark. }Assume that $d\phi=0$ and that $\pi$ is a Riemannian submersion with coassociative fibres. One can then show \cite{PaciniRaffero} that the two conditions $T=0$ and $A=0$ imply $d\psi=0$. It follows that the fibres are minimizing, thus have constant volume. The simplest example is the standard $T^4$-fibration of the torus $T^7=\R^7/\Z^7$. This example has trivial holonomy.

\cite{Baraglia} proves that, for topological reasons, if $M$ is compact and has holonomy $G_2$ (in particular, both its $G_2$ calibrations are closed, ie $d\phi=0$ and $d\psi=0$) then it does not admit global smooth coassociative fibrations.

\section{Application in the K\"ahler setting}\label{s:application}

In previous sections we showed how pluri-potential theory can be used, in the context of immersions, to control various classes of submanifolds. We now want to provide a geometric application in the dual context of Riemannian submersions. We shall use two ingredients which we have already introduced in the previous sections.

\ 

1. Let $M$ be a K\"ahler manifold. Let $K_M:=\Lambda^{n,0}(M)$ denote the holomorphic line bundle of forms of type $(n,0)$, endowed with the induced Hermitian metric $h$. The curvature of $K_M^{-1}$ is the Ricci 2-form $\rho$ of $M$ characterized by the property:
$$\rho(X,Y)=\Ric(JX,Y).$$
It is a closed 2-form of type $(1,1)$. According to general theory, any nowhere-vanishing local holomorphic section $\sigma$ of $K_M^{-1}$ provides a function $H:=h(\sigma,\sigma)$ such that
$$\rho=i\partial\bar\partial (-\log H).$$
The key point for us is the following: $\log H$ is PSH iff $\Ric\leq 0$. 

\ 

\textit{Remark.} Recall that we have already discussed (local) geometric constructions of PSH functions. In the K\"ahler context this is yet another such construction: this time, via a curvature condition on $M$.

\ 

2. Assume $M$ is compact and $G:=T^n$ acts isometrically on $M$. General theory implies that the action also preserves $J$. We shall assume the following: 

(i) The action is free on an open subset of $M$. Let $\mathcal{U}\subseteq M$ denote the maximal such subset. 

(ii) The orbits in $\mathcal{U}$ are Lagrangian. 

\ 

\textit{Example. }The conditions are satisfied for any toric K\"ahler manifold, when the action is Hamiltonian. This is a very large class of examples. 

However, they may be satisfied also by non-Hamiltonian actions, eg the $\Sph^1$-action by rotations on the standard 2-torus. In particular, the Lagrangian condition is always satisfied in dimension 2, ie when $M$ is a Riemann surface.

\ 

Let $\pi:\mathcal{U}\rightarrow B$ denote the corresponding Riemannian submersion, where the fibres are the group orbits. Choose any basis $\{v_j\}$ of the Lie algebra $\mathfrak{g}$ such that $v_1^*\wedge\dots\wedge v_n^*$ coincides with the Haar measure $d\mu$. Let $\tilde{v}_j$ denote the corresponding fundamental vector fields on $\mathcal{U}$. Set $g_{jk}:=g(\tilde{v}_j,\tilde{v}_k)$. The group action implies that these functions are $\pi$-invariant so, along any orbit $\mathcal{O}$, 
$$\Vol(\mathcal{O})=\int_G\sqrt{\det g_{jk}}\,d\mu=\sqrt{\det g_{jk}}.$$ 

The complexified vector fields $\tilde{v}_j-iJ\tilde{v}_j$ are holomorphic. It follows that $\sigma:=(\tilde{v}_1-iJ\tilde{v}_1)\wedge\dots\wedge(\tilde{v}_n-iJ\tilde{v}_n)$ is a nowhere-vanishing holomorphic section of $K_M^{-1}$ over $\mathcal{U}$. Set 
$H:=h(\sigma,\sigma)=\det h_{j\bar{k}}$, where $h_{j\bar{k}}:=g^\C(\tilde{v}_j-iJ\tilde{v}_j,\tilde{v}_k+iJ\tilde{v}_k)$.

The key point now is that the Lagrangian condition implies that $h_{j\bar{k}}=2g_{jk}$, so (up to a constant) the volume of the orbits coincides with $\sqrt{H}$.

\ 

The bottom line is that $H$ has a dual role: on the one hand it provides a potential for the Ricci curvature, on the other it provides the means for calculating the volume of the orbits of the group action. 

This leads to the following result, which (using stronger hypotheses) strongly improves Corollary \ref{cor:volume}: it provides a complete dictionary between potential-theoretic properties of $\Vol$ and the sign of the ambient curvature.

\begin{thm}\label{thm:kahlersubs}
Let $M$ be a compact K\"ahler manifold endowed with an isometric $\T^n$-action. Assume that, on a maximal open subset $\mathcal{U}$, the action is free with Lagrangian orbits. Let $\pi:\mathcal{U}\rightarrow B$ denote the corresponding Riemannian submersion and $\Vol$ the volume functional on $B$, ie the volume of the orbits. Then there exists a minimal orbit $\mathcal{O}_b$ which maximizes $\Vol$ in $B$.

Furthermore, choose any minimal orbit $\mathcal{O}_b$. Then:
\begin{enumerate}
\item For any $X\in T_bB$, the Ricci curvature of $M$ satisfies
$$\Ric(X,X)=\Hess_B(-\log\Vol)(X,X),$$
so $\Ric\leq 0$ $(\Ric\geq 0)$ along $\mathcal{O}_b$ iff $\log\Vol$ $(\log(1/\Vol))$ is convex at $b$.
\item Along $\mathcal{O}_b$, the scalar curvature of $M$ satisfies
$$s=-2(\Delta_B(\Vol)/\Vol)_{|b},$$
so $\pm s\leq 0$ iff $\pm\Vol$ is subharmonic at $b$. 
\item If all fibres are minimal then $\mathcal{U}$ is Ricci-flat.
\end{enumerate}
\end{thm} 
\begin{proof}The maximizing orbit exists by compactness of $M$ (a minimizing orbit might not, because the orbits may collapse outside $\mathcal{U}$). The general theory of orbits \cite{Palais} implies that it is a minimal submanifold.

Part 1: Minimality implies that $\Vol$, thus $H$ and $\log H$, are critical along $\mathcal{O}_b$. Then 
\begin{align*}\Ric(X,X)&=\rho(X,JX)=i\partial\bar\partial(-\log H)(X,JX)\\
&=(1/2)\Hess_B(-\log H)(X,X)\\
&=\Hess_B(-\log\Vol)(X,X),
\end{align*}
where on the second line we use Proposition \ref{p:kahler}. The conclusion follows from the fact that $\Ric(JX,JX)=\Ric(X,X)$.

An alternative proof uses the fact that the induced metric on each orbit is bi-invariant because $\T^n$ is Abelian. In this situation it is known that the fundamental vector fields are parallel and that the induced curvature is zero. Each orbit is thus a flat torus. It follows that, at a minimal orbit, Oh's second variation formula simplifies, becoming 
$$\frac{d^2}{dt^2}\Vol_{|t=0}=\int_{\mathcal{O}_b}-\Ric(X,X)\vol=-\Ric(X,X)\Vol,$$ 
leading to the same result.

Part 2: Choose an ON basis $\{e_j\}$ of $T_p\mathcal{O}_b$. The scalar curvature $s$ can be written
\begin{align*}
s&=\Ric(e_j,e_j)+\Ric(Je_j,Je_j)=2\Ric(e_j,e_j)\\
&=2\rho(e_j,Je_j)=2i\partial\bar\partial(-\log H)(e_j,Je_j)\\
&=\Delta_M(-\log H),
\end{align*}
where in the last line we use the formula in the remark following Proposition \ref{p:kahler}.
Minimality implies that, along that fibre, the vertical contribution to $\Delta_M$ vanishes. We thus find, along that fibre,
\begin{equation*}
s=2\Delta_B(-\log \sqrt H)=2\Delta_B(-\log \Vol)=-2\Delta_B(\Vol)/\Vol,
\end{equation*}
where we also use the fact that $\nabla\Vol=0$.

Part 3: The assumption implies that the volume functional is constant. In turn, this implies that the Ricci potential is constant so $\Ric=0$.
\end{proof}

\ 

\textit{Remark. }We restrict to torus actions simply to take into account the fact (see previous section) that Lagrangian fibrations have torus fibres.

\ 

\textit{Remark. }In general we should not expect $\mathcal{U}=M$. Indeed, this would imply for example that the Ricci 2-form has a global potential. For example, the standard $\Sph^1$-action on $M:=\CP^1$ is not free at the two poles. Correspondingly, the base space becomes singular there.

\

We can use this result to help locate volume maximizing/minimizing orbits.

\begin{cor}\label{cor:kahlersubs}
In the above context:
\begin{enumerate}
\item $\Ric<0$ $(\Ric>0)$ implies that any critical point $b$, ie minimal orbit $\mathcal{O}_b$, is a strict local minimum (maximum) point for $\Vol$.
\item For any local minimum (maximum) point $b\in B$, $\Ric\leq 0$ thus $s\leq 0$ $(\Ric\geq 0, s\geq 0)$ along $\mathcal{O}_b$.
\end{enumerate}
\end{cor}

\ 

\textit{Example. }We can apply these results to any surface in $\R^3$ obtained by rotating a curve $\gamma$ in the $(y,z)$-plane around the $z$-axis. It is a standard exercise, in this case, to show that the rotationally symmetric geodesics correspond to points of $\gamma$ whose tangent line is parallel to the $z$-axis. This leads to the same conclusion regarding curvatures.

\paragraph{Comparisons with previous literature.}Statements very similar to Theorem \ref{thm:kahlersubs} already appear in the literature, in particular in work by Abreu \cite{Abreu} and Oliveira, Sena-Dias \cite{GoncaloRosa}. The main novelty in our presentation is that we contextualize these results in terms of the larger framework of potential theory and Riemannian submersions.

Recall that K\"ahler geometry rests upon very close interactions between symplectic, complex and Riemannian geometry. \cite{Abreu} and \cite{GoncaloRosa} emphasize the symplectic aspects and corresponding coordinate systems. A second point of view, emphasizing complex geometry via pluri-potential theory, appeared in \cite{Pacini}. Theorem \ref{thm:kahlersubs} completes the picture by presenting the Riemannian perspective.

Compared to \cite{Abreu} and \cite{GoncaloRosa}, we have removed the assumption that the action be Hamiltonian. The previous section on Lagrangian fibrations implies however that our action will in any case be locally Hamiltonian, so the main difference lies in our general framework and coordinate-free approach. 

Compared to \cite{Pacini}, notice that both \cite{GoncaloRosa} and Theorem \ref{thm:kahlersubs} discuss convexity only at minimal fibres. The complex viewpoint adopted in \cite{Pacini}, using free $G^c$-actions, leads instead to a global convexity result in terms of the natural structure on $G^c/G$. If $M$ is compact then $\mbox{Aut}(M)$ is a complex Lie group so, by the universal property of complexified Lie groups, any $G$-action, seen as a homomorphism $G\rightarrow\mbox{Aut}(M)$, extends to $G^c$-action on $M$. In this sense, the setting of Theorem \ref{thm:kahlersubs} and \cite{Pacini} are locally similar. Notice however that the complexified action might not be free, eg when we complexify the $\Sph^1$-action on a standard 2-torus. The Riemannian viewpoint adopted here offers, on the other hand, the possibility of defining subharmonicity, which one can then relate to scalar curvature as above.

\bibliographystyle{amsplain}
\bibliography{subPSH_biblio}

\begin{bibdiv}
\begin{biblist}

\bib{Abreu}{article}{
      author={Abreu, Miguel},
       title={K\"{a}hler geometry of toric varieties and extremal metrics},
        date={1998},
        ISSN={0129-167X},
     journal={Internat. J. Math.},
      volume={9},
      number={6},
       pages={641\ndash 651},
         url={https://doi-org.bibliopass.unito.it/10.1142/S0129167X98000282},
      review={\MR{1644291}},
}

\bib{Baraglia}{article}{
      author={Baraglia, D.},
       title={Moduli of coassociative submanifolds and semi-flat
  {$G_2$}-manifolds},
        date={2010},
        ISSN={0393-0440},
     journal={J. Geom. Phys.},
      volume={60},
      number={12},
       pages={1903\ndash 1918},
  url={https://doi-org.bibliopass.unito.it/10.1016/j.geomphys.2010.07.006},
      review={\MR{2735277}},
}

\bib{Bryant}{inproceedings}{
      author={Bryant, Robert~L.},
       title={Some remarks on {$G_2$}-structures},
        date={2006},
   booktitle={Proceedings of {G}\"{o}kova {G}eometry-{T}opology {C}onference
  2005},
   publisher={G\"{o}kova Geometry/Topology Conference (GGT), G\"{o}kova},
       pages={75\ndash 109},
}

\bib{Donaldson}{incollection}{
      author={Donaldson, Simon},
       title={Adiabatic limits of co-associative {K}ovalev-{L}efschetz
  fibrations},
        date={2017},
   booktitle={Algebra, geometry, and physics in the 21st century},
      series={Progr. Math.},
      volume={324},
   publisher={Birkh\"{a}user/Springer, Cham},
       pages={1\ndash 29},
         url={https://doi-org.bibliopass.unito.it/10.1007/978-3-319-59939-7_1},
      review={\MR{3702382}},
}

\bib{EellsSampson}{article}{
      author={Eells, James, Jr.},
      author={Sampson, J.~H.},
       title={Harmonic mappings of {R}iemannian manifolds},
        date={1964},
        ISSN={0002-9327,1080-6377},
     journal={Amer. J. Math.},
      volume={86},
       pages={109\ndash 160},
         url={https://doi.org/10.2307/2373037},
      review={\MR{164306}},
}

\bib{Eschenburg}{article}{
      author={Eschenburg, Jost-Hinrich},
       title={Comparison theorems in riemannian geometry},
     journal={Lecture notes for a course at Trento University, 1994},
}

\bib{Evans}{article}{
      author={Evans, Jonathan~David},
       title={Lectures on lagrangian torus fibrations},
     journal={www.arxiv.org},
}

\bib{Fuglede}{article}{
      author={Fuglede, Bent},
       title={Harmonic morphisms between {R}iemannian manifolds},
        date={1978},
        ISSN={0373-0956,1777-5310},
     journal={Ann. Inst. Fourier (Grenoble)},
      volume={28},
      number={2},
       pages={vi, 107\ndash 144},
         url={http://www.numdam.org/item?id=AIF_1978__28_2_107_0},
      review={\MR{499588}},
}

\bib{GreeneShiohama}{article}{
      author={Greene, R.~E.},
      author={Shiohama, K.},
       title={Convex functions on complete noncompact manifolds: topological
  structure},
        date={1981},
        ISSN={0020-9910,1432-1297},
     journal={Invent. Math.},
      volume={63},
      number={1},
       pages={129\ndash 157},
         url={https://doi.org/10.1007/BF01389196},
      review={\MR{608531}},
}

\bib{GYZ}{article}{
      author={Gukov, Sergei},
      author={Yau, Shing-Tung},
      author={Zaslow, Eric},
       title={Duality and fibrations on {$G_2$} manifolds},
        date={2003},
        ISSN={1300-0098},
     journal={Turkish J. Math.},
      volume={27},
      number={1},
       pages={61\ndash 97},
      review={\MR{1975332}},
}

\bib{HL2}{article}{
      author={Harvey, F.~Reese},
      author={Lawson, H.~Blaine, Jr.},
       title={An introduction to potential theory in calibrated geometry},
        date={2009},
        ISSN={0002-9327},
     journal={Amer. J. Math.},
      volume={131},
      number={4},
       pages={893\ndash 944},
         url={https://doi-org.bibliopass.unito.it/10.1353/ajm.0.0067},
      review={\MR{2543918}},
}

\bib{HL3}{article}{
      author={Harvey, F.~Reese},
      author={Lawson, H.~Blaine, Jr.},
       title={Geometric plurisubharmonicity and convexity: an introduction},
        date={2012},
        ISSN={0001-8708},
     journal={Adv. Math.},
      volume={230},
      number={4-6},
       pages={2428\ndash 2456},
         url={https://doi-org.bibliopass.unito.it/10.1016/j.aim.2012.03.033},
      review={\MR{2927376}},
}

\bib{HarveyWells}{article}{
      author={Harvey, F.~Reese},
      author={Wells, R.~O., Jr.},
       title={Zero sets of non-negative strictly plurisubharmonic functions},
        date={1973},
        ISSN={0025-5831},
     journal={Math. Ann.},
      volume={201},
       pages={165\ndash 170},
         url={https://doi-org.bibliopass.unito.it/10.1007/BF01359794},
      review={\MR{330510}},
}

\bib{HL1}{article}{
      author={Harvey, Reese},
      author={Lawson, H.~Blaine, Jr.},
       title={Calibrated geometries},
        date={1982},
        ISSN={0001-5962},
     journal={Acta Math.},
      volume={148},
       pages={47\ndash 157},
         url={https://doi-org.bibliopass.unito.it/10.1007/BF02392726},
      review={\MR{666108}},
}

\bib{Hermann}{article}{
      author={Hermann, Robert},
       title={A sufficient condition that a mapping of {R}iemannian manifolds
  be a fibre bundle},
        date={1960},
        ISSN={0002-9939},
     journal={Proc. Amer. Math. Soc.},
      volume={11},
       pages={236\ndash 242},
         url={https://doi-org.bibliopass.unito.it/10.2307/2032963},
      review={\MR{112151}},
}

\bib{G2lectures}{book}{
      editor={Karigiannis, Spiro},
      editor={Leung, Naichung~Conan},
      editor={Lotay, Jason~D.},
       title={Lectures and surveys on g2-manifolds and related topics},
      series={Fields Institute Communications},
   publisher={Springer, New York},
        date={[2020] \copyright 2020},
      volume={84},
        ISBN={978-1-0716-0576-9; 978-1-0716-0577-6},
         url={https://doi-org.bibliopass.unito.it/10.1007/978-1-0716-0577-6},
        note={Including lectures from the Minischool and Workshop on
  G2-Manifolds held as part of the Major Thematic Program on Geometric Analysis
  at the Fields Institute, Toronto, August 19--25, 2017},
      review={\MR{4295851}},
}

\bib{JasonSpiro}{article}{
      author={Karigiannis, Spiro},
      author={Lotay, Jason~D.},
       title={Bryant-{S}alamon {$\rm G_2$} manifolds and coassociative
  fibrations},
        date={2021},
        ISSN={0393-0440},
     journal={J. Geom. Phys.},
      volume={162},
       pages={Paper No. 104074, 60},
  url={https://doi-org.bibliopass.unito.it/10.1016/j.geomphys.2020.104074},
      review={\MR{4199397}},
}

\bib{Lassalle}{article}{
      author={Lassalle, Michel},
       title={Deux g\'{e}n\'{e}ralisations du ``th\'{e}or\`eme des trois
  cercles'' de {H}adamard},
        date={1980},
        ISSN={0025-5831},
     journal={Math. Ann.},
      volume={249},
      number={1},
       pages={17\ndash 26},
         url={https://doi-org.bibliopass.unito.it/10.1007/BF01387077},
      review={\MR{575445}},
}

\bib{Lelong}{book}{
      author={Lelong, Pierre},
       title={Fonctionnelles analytiques et fonctions enti\`eres ({$n$}\
  variables)},
      series={S\'{e}minaire de Math\'{e}matiques Sup\'{e}rieures, No. 13
  (\'{E}t\'{e}},
   publisher={Les Presses de l'Universit\'{e} de Montr\'{e}al, Montreal, Que.},
        date={1968},
      volume={1967},
      review={\MR{0466606}},
}

\bib{Nagano}{article}{
      author={Nagano, Tadashi},
       title={On fibred {R}iemann manifolds},
        date={1960},
        ISSN={0040-8964},
     journal={Sci. Papers College Gen. Ed. Univ. Tokyo},
      volume={10},
       pages={17\ndash 27},
      review={\MR{157325}},
}

\bib{Oh}{article}{
      author={Oh, Yong-Geun},
       title={Second variation and stabilities of minimal {L}agrangian
  submanifolds in {K}\"{a}hler manifolds},
        date={1990},
        ISSN={0020-9910},
     journal={Invent. Math.},
      volume={101},
      number={2},
       pages={501\ndash 519},
         url={https://doi-org.bibliopass.unito.it/10.1007/BF01231513},
      review={\MR{1062973}},
}

\bib{GoncaloRosa}{article}{
      author={Oliveira, Gon\c{c}alo},
      author={Sena-Dias, Rosa},
       title={Minimal {L}agrangian tori and action-angle coordinates},
        date={2021},
        ISSN={0002-9947},
     journal={Trans. Amer. Math. Soc.},
      volume={374},
      number={11},
       pages={7715\ndash 7742},
         url={https://doi-org.bibliopass.unito.it/10.1090/tran/8403},
      review={\MR{4328681}},
}

\bib{Oneill}{article}{
      author={O'Neill, Barrett},
       title={The fundamental equations of a submersion},
        date={1966},
        ISSN={0026-2285,1945-2365},
     journal={Michigan Math. J.},
      volume={13},
       pages={459\ndash 469},
         url={http://projecteuclid.org/euclid.mmj/1028999604},
      review={\MR{200865}},
}

\bib{Pacini}{article}{
      author={Pacini, Tommaso},
       title={Ricci curvature, the convexity of volume and minimal lagrangian
  submanifolds},
     journal={to appear in J. Sympl. Geometry},
}

\bib{PaciniRaffero}{article}{
      author={Pacini, Tommaso},
      author={Raffero, Alberto},
       title={Variation formulae for the volume of coassociative submanifolds},
     journal={www.arxiv.org},
}

\bib{Palais}{article}{
      author={Palais, Richard~S.},
       title={The principle of symmetric criticality},
        date={1979},
        ISSN={0010-3616},
     journal={Comm. Math. Phys.},
      volume={69},
      number={1},
       pages={19\ndash 30},
  url={http://projecteuclid.org.bibliopass.unito.it/euclid.cmp/1103905401},
      review={\MR{547524}},
}

\bib{Petersen}{book}{
      author={Petersen, Peter},
       title={Riemannian geometry},
     edition={Third},
      series={Graduate Texts in Mathematics},
   publisher={Springer, Cham},
        date={2016},
      volume={171},
        ISBN={978-3-319-26652-7; 978-3-319-26654-1},
         url={https://doi-org.bibliopass.unito.it/10.1007/978-3-319-26654-1},
      review={\MR{3469435}},
}

\bib{Sakai}{book}{
      author={Sakai, Takashi},
       title={Riemannian geometry},
      series={Translations of Mathematical Monographs},
   publisher={American Mathematical Society, Providence, RI},
        date={1996},
      volume={149},
        ISBN={0-8218-0284-4},
         url={https://doi.org/10.1090/mmono/149},
        note={Translated from the 1992 Japanese original by the author},
      review={\MR{1390760}},
}

\bib{Shiga}{incollection}{
      author={Shiga, Kiyoshi},
       title={Hadamard manifolds},
        date={1984},
   booktitle={Geometry of geodesics and related topics ({T}okyo, 1982)},
      series={Adv. Stud. Pure Math.},
      volume={3},
   publisher={North-Holland, Amsterdam},
       pages={239\ndash 281},
         url={https://doi.org/10.2969/aspm/00310239},
      review={\MR{758657}},
}

\bib{Simons}{article}{
      author={Simons, James},
       title={Minimal varieties in riemannian manifolds},
        date={1968},
        ISSN={0003-486X},
     journal={Ann. of Math. (2)},
      volume={88},
       pages={62\ndash 105},
         url={https://doi-org.bibliopass.unito.it/10.2307/1970556},
      review={\MR{233295}},
}

\bib{SYZ}{article}{
      author={Strominger, Andrew},
      author={Yau, Shing-Tung},
      author={Zaslow, Eric},
       title={Mirror symmetry is {$T$}-duality},
        date={1996},
        ISSN={0550-3213},
     journal={Nuclear Phys. B},
      volume={479},
      number={1-2},
       pages={243\ndash 259},
  url={https://doi-org.bibliopass.unito.it/10.1016/0550-3213(96)00434-8},
      review={\MR{1429831}},
}

\bib{TsaiWang}{article}{
      author={Tsai, Chung-Jun},
      author={Wang, Mu-Tao},
       title={Mean curvature flows in manifolds of special holonomy},
        date={2018},
        ISSN={0022-040X},
     journal={J. Differential Geom.},
      volume={108},
      number={3},
       pages={531\ndash 569},
         url={https://doi-org.bibliopass.unito.it/10.4310/jdg/1519959625},
      review={\MR{3770850}},
}

\end{biblist}
\end{bibdiv}

\end{document}